\documentclass[a4paper, 11pt, cleardoublepage=empty,bibliography=totoc]{scrartcl} 
\usepackage[left=2.5cm, right=2.5cm, top=0.5cm, bottom=2.5cm]{geometry}
\usepackage[colorlinks,
pdfpagelabels,
pdfstartview = FitH,
bookmarksopen = true,
bookmarksnumbered = true,
linkcolor = black,
plainpages = false,
hypertexnames = false,
citecolor = black] {hyperref}
\usepackage{adjustbox}
\usepackage{amsmath,amssymb,mathrsfs,amsfonts}

\usepackage{bm}
\newcommand*{\B}[1]{\ifmmode\bm{#1}\else\textbf{#1}\fi}
\usepackage[english]{babel}
\usepackage{tikz-cd}
\usepackage{amstext}
\usepackage{pdfpages}
\usepackage{mathtools}
\usepackage{stmaryrd}
\usepackage{xcolor}
\usepackage{xfrac} 
\usepackage[utf8]{inputenc}
\usepackage{url} 
\usepackage{pgfplots}
\usepackage[autooneside=false,headsepline,markcase=noupper]{scrlayer-scrpage}
\usepackage{blindtext}
\usepackage{enumitem}
\usepackage{verbatim}
\usepackage{graphicx}
\usepackage{mathtools}
\usepackage{makeidx}
\usepackage{tikz}
\usetikzlibrary{babel}

\addto\captionsenglish{}

\usetikzlibrary{%
	matrix,%
	calc,%
	arrows%
}

\definecolor{mycolor}{RGB}{196,19,47}
\definecolor{mygray}{gray}{0.4}

\setkomafont{section}{\normalfont\LARGE\bfseries}
\setkomafont{subsection}{\normalfont\large\bfseries}
\setkomafont{pagehead}{\normalfont\bfseries}
\setkomafont{title}{\normalfont\bfseries\Large}

\usepackage{chngcntr}
\counterwithin{equation}{subsection}
\usepackage{amsthm}

\theoremstyle{remark}
\newtheorem{remark}[equation]{Remark}

\theoremstyle{definition}
\newtheorem{df}[equation]{Definition}
\newtheorem{cons}[equation]{Construction}
\newtheorem{obs}[equation]{Observation}
\newtheorem{Zusa}[equation]{Summary}
\newtheorem*{ass}{Assumption $(*)$}
\newtheorem{ctr}[equation]{Rule}

\newtheorem{ex}[equation]{Example}

\theoremstyle{plain}

\newtheorem{theorem}[equation]{Theorem}

\newtheorem{corollary}[equation]{Corollary}

\newtheorem{lemma}[equation]{Lemma}

\newtheorem*{thm_nn_a}{Theorem A}

\newtheorem*{l_b}{Lemma B}

\newtheorem*{thm_nn_c}{Theorem C}

\newtheorem*{ack}{Acknowledgements}
\clearpairofpagestyles 
\ohead{  \ifstr{\rightbotmark}{\leftmark}{}{\rightbotmark}}
\ihead{\leftmark}
\rohead*{\pagemark}
\rehead*{\pagemark}
\lehead{\headmark}
\lohead{\headmark}

\clearscrheadfoot
\cofoot[\pagemark]{\pagemark}

\makeatletter

\providecommand*{\rightbotmark}{\expandafter\@rightmark\botmark\@empty\@empty}
\makeatother
\makeindex

\usepackage{mathabx}

\usepackage{amssymb}

\begin{document}
\pagestyle{plain}
\title{Open-Closed Trivialization of Circle Action for Calabi-Yau }
\begin{center}
    
\large{\textbf{Towards Open-Closed Categorical Enumerative Invariants: Circle-Action Formality Morphism}}
\end{center}
\hspace{0.5cm}
\begin{center}
    Jakob Ulmer
  \end{center}
\hspace{1cm}

\textbf{Abstract.} Categorical enumerative invariants of a Calabi-Yau category, encoded as the partition function of the associated closed string field theory (SFT), conjecturally equal Gromov-Witten invariants when applied to Fukaya categories. Part of this theory is a formality $L_\infty$-morphism which depends on a splitting of the non-commutative Hodge filtration.
Our main result is providing an open-closed formality morphism; the algebraic structures involved conjecturally give a home to {open-closed} GW invariants. We explain how the open-closed morphism is an ingredient towards quantizing the large $N$ open SFT of an object of a Calabi-Yau category. 
\tableofcontents
\section{Introduction}
Physicists introduced topological string theory, which `probes' a given Calabi-Yau manifold and allows to extract numerical invariants \cite{can91}.
A topological string theory comes in open, closed and open-closed versions.
For the closed topological string theory called the (closed) A-model those invariants have been developed under the name (closed) Gromov-Witten invariants and are well studied by now.

Kontsevich introduced a categorical version of Calabi-Yau spaces \cite{Kon94}, which had a profound influence in mathematics and in the field of mirror symmetry. It turns out that any open-closed topological string can be fully understood by a Calabi-Yau category \cite{Cos07a} \cite{Lu09}.
This raises the question if we can define an analogue of mentioned enumerative invariants for such categories so that they coincide with the (closed) Gromov-Witten invariants when applied to the Fukyay category of a symplectic manifold.
In section 1.1 we recall a theory which proposes an answer, developed by Costello \cite{Cos05}, Caldararu-Tu \cite{CaTu24}, following ideas of Kontsevich. Physically this is about passing to the closed string field theory (SFT) associated to the topological string theory. To be precise their procedure requires an additional ingredient, a splitting of the non-commutative Hodge filtration:
 A splitting of the Hodge filtration should roughly allow to extend the topological string theory to a natural compactification, inducing a cohomological field theory, compare \cite{des22}.
\begin{equation}\label{diagr}
\begin{tikzcd}
    \text{Calabi-Yau category}\arrow[bend left=25]{r}{\text{splitting of Hodge filtration}}&\text{`enumerative invariants'}
\end{tikzcd}\end{equation}
 
 In \cite{CaTu24} such a splitting is used to construct a formality morphism, which is a key step allowing them to extract numerical invariants, as we review below.

 Physicists also propose to extract invariants form open-closed topological string theories, which are mathematically more subtle to define, however. Nevertheless one can ask whether a categorical generalization exists, in the same vein as explained for the closed theory. In this note we provide an open-closed version of the formality morphism, again depending on a splitting of the non-commutative Hodge filtration (theorem A).
 To do so we introduce the algebraic structures aiming to provide a home to open-closed Gromov-Witten invariants. Physically, this is about considering the (large $N$) open-closed SFT associated to the Calabi-Yau category and a choice of objects. We explain how the open-closed formality morphism is helpful in finding a quantization of the (large $N$) open SFT (theorem C).
 
 Let us go through these ideas and results in more detail:
\subsection{Closed Categorical Enumerative Invariants}
Associated to a dimension $d$ Calabi Yau category $\mathcal{C}$ is a chain complex $\mathcal{F}^c(\mathcal{C})$ built\footnote{In fact we perform some shifts compared to what is shown here- see definition \ref{cBD}. We do so in order to be compatible with the open BD algebra \ref{oBD}. }  from its Hochschild chains 
$\big(CH_*(\mathcal{C}),d_{hoch}\big)$, which carries the structure of a $(2d-5)$ twisted Beilinson-Drinfeld algebra
\begin{equation}\label{cBDi}
\mathcal{F}^c(\mathcal{C}):=\Big(Sym\big(CH_*(\mathcal{C})[ u^{-1}]\big)\llbracket\gamma\rrbracket,d_{hoch}+uB+\gamma\Delta_c,\{\_,\_\}_c\Big),\end{equation} the 
\emph{(free) quantum observables of the closed string field theory associated to $\mathcal{C}$}.\footnote{See \cite{CL15} where this name is justified from the perspective of physics in the case of the category of coherent sheaves, called the B-model, which leads to Kodaira-Spencer gravity or BCOV theory, known as the string field theory of the B-model.} Its shifted Poisson structure $\{\_,\_\}_c$ and the free BV Laplacian $\Delta_c$ are built from the so called Mukai pairing $\langle\_,\_\rangle_M$ and the $S^1$ action on Hochschild chains, encoded by Connes operator $B$. Roughly speaking Kaledin proved \cite{kal17} that given a smooth and proper $\mathbb{Z}$-graded category the $S^1$-action on Hochschild chains is trivial.
A splitting of the non-commutative Hodge filtration specifies \emph{how} this action is trivial. 
Given such we can trivialize most of the algebraic structure of $\mathcal{F}(\mathcal{C})$.
Denote by $\mathcal{F}(\mathcal{C})^{Triv}$ the same underlying graded vector space, but where all operations except $d_{hoch}$ are zero.
\begin{theorem}[\cite{CaTu24}, section 7, \cite{AmTu25}]
Given a proper Calabi-Yau $A_\infty$-category $\mathcal{C}$ and a splitting~$s$ then there is an $L_\infty$ quasi-isomorphism of induced \footnote{by forgetting the multiplication and shifting appropriately} dg Lie algebras
\begin{equation}\label{cltr}
\Psi_s^c:\ \mathcal{F}^c(\mathcal{C})\rightsquigarrow {\mathcal{F}^c(\mathcal{C})}^{Triv}.\end{equation}
\end{theorem}
Another fundamental ingredient is
\begin{theorem}[\cite{CaTu24, Cos05}, \cite{AmTu25}]\label{cGWP}
Given a smooth and proper Calabi-Yau $A_\infty$-category $\mathcal{C}$ there is a Maurer-Cartan element, determined under the TCFT action \cite{Cos07a} by the closed string vertices \cite{SZ94}, denoted
\begin{equation*}
I^{c,q}\in MCE(\mathcal{F}^c(\mathcal{C})).
\end{equation*}
 \end{theorem}
The (conjectural enumerative) invariants from diagram \ref{diagr} are encoded as the image of the Maurer-Cartan element $I^{c,q}$ under the map \eqref{cltr}.  More precisely under the isomorphism 
$$H_*({\mathcal{F}^c(\mathcal{C})}^{Triv})\cong Sym (HH_*(\mathcal{C})[u^{-1}])$$ we read off the coefficients of the determined homology class with respect to Hochschild homology-class and `psi-class' insertions, counted by the powers of $u$. 
\begin{remark}
    To give an actual formula for these coefficients is the main accomplishment of \cite{CaTu24}, see theorem 9.1. It seems however that in general one can do so only in an indirect way. To explicitly compute these numbers for special cases and in low genus is extremely hard, see section 9.2 of \cite{CaTu24} and \cite{catu17}.
\end{remark}
\subsection{Towards Open-Closed Categorical Enumerative Invariants}
Physicists and symplectic topologists (eg. \cite{psw08}) have been considering \emph{open} Gromov Witten invariants, which count maps from Riemann surfaces with boundary into a symplectic manifold such that the boundar(ies) lie in (a) fixed Lagrangian submanifold(s). Such invariants are subtle to define as the moduli space of such maps has an intricate boundary behavior \cite{Liu02}.
On the other hand, in \cite{EkSh25} the authors use exactly the structures induced from this boundary behavior to show that open Gromov Witten invariants have a natural home in the skein module.
Earlier Fukaya \cite{Fuk11} studied algebraic structures which similarly exhibit a home for such invariants in genus zero, a higher genus version was perhaps first described in \cite{mo99}; those structures also appear in string topology, eg. \cite{NaWi19}.

Following these insights, a way to assemble categorical open-closed invariants may be formulated as follows: Algebraically we now fix not just a category $\mathcal{C}$ as before \footnote{which we in fact now assume not just to be proper Calabi-Yau, but cyclic, which is a strictification of the first notion.} but also a set of objects $\Lambda$, which would correspond to Lagrangian submanifolds as objects of the Fukaya category. Then the natural open-closed analogue of \eqref{cBDi} is the tensor product BD algebra
$$\mathcal{F}^c(\mathcal{C})\otimes \mathcal{F}^o(\Lambda),$$
where $$\mathcal{F}^o(\Lambda)=\Big(Sym\big(Cyc^*(\Lambda)[-1]\big) \llbracket\gamma\rrbracket,d+\nabla+\gamma\delta,\{\_,\_\}_o\Big),$$ is the $(2d-5)$ twisted Beilinson-Drinfeld algebra induced from the involutive Lie bi-algebra structure on (a central extension of) cyclic cochains $Cyc^*(\Lambda)$. $\mathcal{F}^o(\Lambda)$ is further the domain of the quantized Loday-Quillen-Tsygan map \cite{GGHZ21, Ul25a}, which justifies calling $\mathcal{F}^o(\Lambda)$ \emph{the large $N$ (free) quantum observables of the open string field theory on a stack of $N$ branes of type~$\Lambda$}.\footnote{Compare \cite{CL15} where this is related to holomorphic Chern-Simons theory in the case of the B-model, by taking the `cyclic' $A_\infty$-algebra to be Dolbeault forms on $\mathbb{C}^3$.}

The main work in this note is about proving the following:
\begin{thm_nn_a}[Corollary \ref{maco}]
Given a dimension $d$ cyclic $A_\infty$-category $\mathcal{C}$, a set of objects $\Lambda$ and a splitting $s$ of the non-commutative Hodge filtration. Then there is an $L_\infty$ quasi-isomorphism
\begin{equation}\label{octr}
\Psi_s^{oc}:\ \mathcal{F}^{c}(\mathcal{C})\otimes \mathcal{F}^{o}(\Lambda)\rightsquigarrow \mathcal{F}^{c}(\mathcal{C})^{Triv}\otimes \mathcal{F}^{o}(\Lambda).
\end{equation}
\end{thm_nn_a}
\begin{remark}
The proof of this theorem reduces to proving an identity (theorem \ref{BVinf}) for the Taylor components of the purely closed $L_\infty$-morphism \eqref{cltr}. It would be interesting to have an operadic interpretation of this identity, perhaps in terms of $BV_\infty$-algebras \cite{GTV12}. 
\end{remark}
\begin{remark}
As is typical in this are of mathematical physics the morphism \eqref{octr} is defined in terms of graphs. To verify theorem A (respectively theorem \ref{BVinf}) we prove an interesting identity about the `moduli space of graphs', which is about relating graphs with varying number of vertices to each other, with the additional choice of a partition of the half-edges of a given vertex. This part may be of independent interest.
\end{remark}
We finish this introduction by explaining with  theorem C below how theorem A provides a building stone in quantizing the (large $N$) open SFT of a smooth cyclic $A_\infty$-category on a stack of $N$ branes of type $\Lambda$, for simplicity taking $\Lambda=\{\lambda\}$. Following theorem C we make some remarks about the relationship to the theory of open categorical invariants. 
\begin{ass}
Given a smooth cyclic $A_\infty$-category $\mathcal{C}$ and an object ${\lambda\in \mathcal{C}}$ there is an element $S^{oc}\in \mathcal{F}^c(\mathcal{C})\otimes \mathcal{F}^o(\lambda)$ such that\footnote{Here we denote $d_{oc}=d+d_{Hoch}+uB$, $\Delta_{oc}=\nabla+\gamma\delta+\Delta_c$ and anolougously for the bracket. Further $\nu\langle ch(\lambda),\_\rangle_{M}$ denotes pairing with the non-commutative Chern character $ch(\lambda)$, defined as the class of the unit of $End(\lambda)$ in Hochschild chains, via the Mukai pairing $\langle \_,\_\rangle_{M}$ (\cite{CaTu24}, section 3.4), extended as a derivation to the symmetric algebra, further weighted by $\nu$, the basis of cyclic cochains of length zero. $\nu$ gets mapped to the number $N$ under the quantized LQT map of rank $N$.}
\begin{equation}\label{Ass}
(d_{oc}+\Delta_{oc})(S^{oc})+\{S^{oc},S^{oc}\}_{oc}+\nu\langle ch(\lambda),S^{oc}\rangle_{M}=0,
\end{equation}
that it is a Maurer-Cartan element up to the last term, and the open tree level part\footnote{That is the projection to $Sym^0(\cdots)\otimes Sym^1(Cyc^*(\lambda)[-1])\gamma^0$.} of $S^{oc}$ is \begin{equation}\label{cycpot}
I:=\sum_{k=1}^\infty\langle m^k(\_,\cdots,\_),\_\rangle\in Cyc^*(\lambda)[-1],\end{equation} the cyclic potential\footnote{Compare eg. lemma 3.1.37 of \cite{Ul25a}.} of the cyclic $A_\infty$-algebra $End_\mathcal{C}(\lambda)$, the $m^k$ denoting its structure maps.

\end{ass}
Fundamental work of \cite{Zw98}, \cite{HVZ08} about the open-closed string vertices together with work in preparation \cite{AmTu25} should verify this assumption in general. The purely `open' part of the work \cite{AmTu25} (more precisely the open-closed analogue of the dotted arrow in \cite{CaTu24}) should reduce to the middle vertical arrow of theorem E in \cite{Ul25a}, related to Kontsevich's cocycle construction \cite{Kon92b}.  
\begin{l_b}
    Let $\mathcal{C}$ be a smooth cyclic $A_\infty$-category and ${\lambda\in \mathcal{C}}$ such that assumption $(*)$ holds and such that $ch(\lambda)=dh$. Let 
    $$I^{oc}(\lambda,h):=e^{\nu\langle h,-\rangle_{M}}S^{oc}.$$ Then we have that $$I^{oc}(\lambda,h)\in MCE\big(\mathcal{F}^c(\mathcal{C})\otimes \mathcal{F}^o(\lambda)\big).$$
\end{l_b}
\begin{proof}
This follows by computing
\begin{align*}
&d_{oc}\Big(e^{\nu\langle h,-\rangle_{M}}S^{oc}\Big)=e^{\nu\langle h,-\rangle_{M}}d_{oc}S^{oc}+e^{\nu\langle h,-\rangle_{M}}\nu\langle ch(\lambda),-\rangle_{M}S^{oc}\\
=&-\Big(\Delta_{oc}I^{oc}+\{I^{oc},I^{oc}\}_{oc}+\nu\langle ch(\lambda),e^{\nu\langle h,-\rangle_{M}}S^{oc}\rangle_{M}\Big)+\nu\langle ch(\lambda),-\rangle_{M}e^{\nu\langle h,-\rangle_{M}}S^{oc}\\
=&-\big(\Delta_{oc}I^{oc}+\{I^{oc},I^{oc}\}_{oc}\big),\end{align*}
where the first equality uses lemma 6.0.3 (3) of \cite{Cos05} and for the second equality we used assumption $(*)$ and that $e^{\nu\langle h,-\rangle_{M}}$ commutes with the brackets and BV differentials.\end{proof}
\begin{thm_nn_c}
   Given a smooth cyclic $A_\infty$-category $\mathcal{C}$ and ${\lambda\in \mathcal{C}}$ such that assumption $(*)$ holds, such that $ch(\lambda)=dh$ and given a splitting of the non commutative Hodge filtration $s$ then there exist a quantization $I^q(\lambda,h,s)\in MCE(\mathcal{F}^o(\lambda))$ of the cyclic $A_\infty$ algebra $End_\mathcal{C}(\lambda)$. 
\end{thm_nn_c}
\begin{proof}
Indeed, as the map $$pr:\mathcal{F}^{c}(\mathcal{C})^{Triv}\otimes \mathcal{F}^{o}(\Lambda)\rightarrow Sym^0(\cdot)\otimes \mathcal{F}^o(\lambda)\cong \mathcal{F}^o(\lambda)$$
is a map of dg Lie algebras it follows from the previous lemma B and theorem A that 
$$I^q(\lambda,h,s):=pr\circ\Psi_s^{oc}(I^{oc}(\lambda,h))\in MCE(\mathcal{F}^o(\lambda)).$$ Further by the second part of assumption $(*)$ and the explicit form of the map \ref{octr} one concludes that the dequantization\footnote{That is the projection to $Sym^1(Cyc^*(\lambda)[-1])\gamma^0$.} of $I^q(\lambda,h,s)$ is the cyclic $A_\infty$ potenial of $End_\mathcal{C}(\lambda)$.
\end{proof}
\begin{remark}
One could extract actual numbers from the closed GW analogue of the Maurer-Cartan element $I^q(\lambda,h,s)$, see beneath theorem \ref{cGWP}. We have not explored how to do that here. In open Gromov-Witten theory one should impose some extra conditions on the Lagrangians in order to do that: Eg. in \cite{Fuk11} it is required that the Lagrangian has the homology of a sphere and in \cite{psw08} that the Lagrangian is the locus of an involution. One could ponder similar constraints for the categorical set-up.    
\end{remark}
\begin{remark}
The `non-commutative' BV formalism \cite{Ba10b,CL15,GGHZ21, Ul25a} relates quantizations of Calabi-Yau categories to large $N$ matrix models, the zero dimensional cousins of large $N$ gauge theories. It is known that those describe enumerative invariants, eg. the Kontsevich matrix models \cite{Kon92a} for intersection numbers or its r-spin cousin \cite{BCEGF23}, large $N$ Chern-Simons theory \cite{GoVa99} for GW-invariants of the resolved conifold. We hope to relate these concrete large $N$ gauge theories to the general theory outlined here in future work.
\end{remark}
\begin{remark}
    In \cite{CL15}, page 3 it is argued that open string field theory in general and (large $N$) homolorphic Chern-Simons theory on $\mathbb{C}^3$ in particular does not admit a quantization by itself, in apparent\footnote{Note that we can't apply theorem~C as Dolbeault forms on $\mathbb{C}^3$ do not verify the properness assumption; compare also footnote 5. See however the exciting recent work \cite{ha25a,ha25b} and also \cite{CL15}, which may allow to loosen the properness restriction at the cost of introducing more analysis.} contrast to what we claim here. However note that the quantizations we produce in theorem C have `contributions' from the closed sector by stable graphs with zero leaves and thus do depend on both the open and closed sectors, which is actually similar in spirit to the main idea of \cite{CL15}.
\end{remark}
\begin{ack}
I would like to thank my advisors Owen Gwilliam and Grégory Ginot for support and feedback. I am grateful for the continuous interest of Surya Raghavendran and Philsang Yoo. I would like to thank Junwu Tu for discussions around ideas in this paper and the invitation to come to Shanghai in summer 2024. Lemma B and theorem C took their form during that visit and in joint conversations. My PhD is founded by the European Union’s Horizon 2020 research and innovation programme under the Marie Skłodowska-Curie grant agreement No~945332.    
\end{ack}
\section{Algebraic Preliminaries}
We fix $k$ a field of characteristic 0.
Furthermore we work in the category of chain complexes over $k$, denoted $Ch$, and stick to the homological convention.
The suspension $[\_]$ of a chain complex $V$ is defined in such a way that if $V$ is concentrated in degree zero, then V[1] is concentrated in degree $-1$.
The results of this paper also apply to the $\mathbb{Z}/{2}$-graded context.
\subsection{Beilinson-Drinfeld (BD) Algebras}
We have in mind that $r$ denotes an odd integer\footnote{Since then we understand the quantized LQT map.}, even though we don't need to make this restriction for this paper.
\begin{df}\label{BD}
An $r$-twisted Beilinson-Drinfeld ($BD$) algebra is a graded commutative unital algebra over the ring $k\llbracket \mu\rrbracket$ of formal power series in a variable $\mu$, which has cohomological degree $1-r$. It is endowed with a Poisson bracket $\{\_,\_\}$ of degree $r$, which is $R\llbracket \mu\rrbracket$-linear. Further it is endowed with an $R\llbracket\mu\rrbracket$-linear map $d$ of degree 1 that squares to zero and such that
\begin{equation}\label{BDrel}
d(a\cdot b)=d(a)\cdot b+(-1)^{|a|}a\cdot d(b)+\mu \{a,b\}.
\end{equation}
\end{df}
\begin{remark}
 In \cite{CPTTV17}, section 3.5.1, the notion of a $BD_d$ algebra is introduced. It is not clear to the author what is the relationship of this notion and ours. However, for $s=1$ an $s$-twisted BD algebra is the same as a $BD_0$ algebra.
 \end{remark}
\begin{lemma}\label{tensorBD}
    Given two $r$-twisted BD algebras $A$, $B$ the tensor product of their underlying graded vector spaces $A\otimes B$ can be equipped with the structure of an $r$-twisted BD algebra in a natural way, extending the initial ones.
\end{lemma}
To define the shifted Poisson bracket on the tensor product we use both the respective algebra multiplications, denoted $m_A$ and $m_B$, and the original shifted Poisson brackets. Precisely we have that 
\begin{align}\label{tenbra}
\{a_1\otimes b_1,a_2\otimes b_2\}=(-1)^{|a_2|(|b_1|+r)}m_A(a_1,a_2)&\otimes\{b_1,b_2\}_B\\
&+(-1)^{|b_1|(|a_2|+r)}\{a_1,a_2\}_A\otimes m_B(b_1,b_2).\nonumber
\end{align}

\begin{proof}
This is the same argument described in definition 2.7 and proposition-definition 2.2 of \cite{legu14}, following \cite{man99}, section 5.8.1. Note however that we follow the convention of remark 2.3 (2) of \cite{legu14}.  
\end{proof}

\subsection{Calabi-Yau $A_\infty$-Categories}
An $A_\infty$-category is a homotopically better behaved generalization of the notion of a dg-category, further relevant since Fukaya categories of symplectic manifolds are \emph{not} dg categories, but do carry an $A_\infty$-category structure. We follow the conventions for $A_\infty$-categories from definition 3.1.2 of \cite{Ul25a}. 
An important invariant of $A_\infty$-categories is its Hochschild homology.  
\begin{df}
Let $\mathcal{C}$ be an $A_\infty$-category. Then denote by
$$CH_*(\mathcal{C}):=\mathcal{C}_\Delta\otimes^{\mathbb{L}}_{\mathcal{C}^e}\mathcal{C}_\Delta\in Ch,$$
 \emph{the Hochschild chains} of $\mathcal{C}$, that is the derived tensor product of the diagonal bimodule $\mathcal{C}_\Delta$ in the category of $\mathcal{C}$-bimodules (see eg. \cite{Ga19}, section 3.1, section 3.4 and references therein). Its homology is called \emph{Hochschild homology}.
\end{df}
One of its important properties is that there is an $S^1$-action on Hochschild chains (for algebras discovered by \cite{co85}, in our context eg. \cite{Ga23}, section 3.2). By considering the homotopy $S^1$-orbits of Hochschild chains we arrive at the notion of cyclic homology.
\begin{df}
   Given an $A_\infty$-category $\mathcal{C}$ we call the homology of $$CC_*(\mathcal{C}):=CH_*(\mathcal{C})_{hS^1}\in Ch,$$ the homology of the $S^1$-orbits of Hochschild chains, its \emph{cyclic homology}. 
\end{df}
Concretely we use the following chain complex to compute cyclic homology: $$CC_*(\mathcal{C})\simeq\big(CH_*(\mathcal{C})[ u^{-1}], d_{Hoch}+uB\big).$$
Here $CH_*(\mathcal{C})$ is the graded vector space on which the differential $d_{Hoch}$ computes Hochschild homology, $u^{-1}$ is a formal variable of degree $2$ and $B:CH_*(\mathcal{C})\rightarrow CH_*(\mathcal{C})$ is a chain map of degree 1 called Connes operator, so that $|uB|=-1$.
\\
\\
For the purposes of this paper we consider $A_\infty$-categories with an extra structure, which most generally may be phrased as requiring a proper Calabi-Yau structure. 
We work with the convenient explicit notion of a cyclic structure, which is a  strictification of proper Calabi-Yau and leave it to further investigations how to understand our results at this generality. 
Compare however \cite{choSa12} \cite{BrDy21}   for relevant results in this direction and \cite{AmTu22} for an answer to such investigations in the purely `closed' setting.

We describe the relevant notions to provide some context: We recall that an $A_\infty$-category $\mathcal{C}$ is called \emph{proper} if $Hom_\mathcal{C}(x,y)\in Ch$ is perfect for all $x,y\in Ob(\mathcal{C}).$ 
\begin{df}
Given a proper $A_\infty$-category $\mathcal{C}$ a \emph{weak proper Calabi-Yau structure} of dimension $d$ is a quasi-isomorphism of $\mathcal{C}$-bimodules between the diagonal bimodule $\mathcal{C}_\Delta$ and the linear dual diagonal bimodule $\mathcal{C}_\Delta^\vee$ (see \cite{Ga19}, section 3.1, references therein and section 3.4)
$$\mathcal{C}_\Delta \simeq\mathcal{C}_\Delta^\vee[-d].$$

A weak proper Calabi-Yau structure of dimension $d$, also called weak right Calabi-Yau structure, may be equivalently described by saying that there is a map of complexes $$tr: CH_*(\mathcal{C})\rightarrow k[-d]$$ such that the induced map 
$$\begin{tikzcd}
    Hom_\mathcal{C}(x,y)\otimes Hom_\mathcal{C}(y,x)\arrow{r}{m^2_{x,y}}& Hom_\mathcal{C}(x,x)\arrow{r}{\iota_x}& CH_*(\mathcal{C})\arrow{r}{tr}& k[-d]\end{tikzcd}$$
is non-degenerate on homology, for all objects $x,y\in \mathcal{C}$ (see remark 49 of \cite{Ga23}). If the map $tr$ is induced from a map from cyclic chains along the natural map
$CH_*(\mathcal{C})\rightarrow CC_*(\mathcal{C})$ this is called a \emph{strong proper Calabi-Yau structure}.
\end{df}
\begin{df}
    A \emph{cyclic $A_\infty$-category} of degree $d$, denoted $(\mathcal{C},\langle\_,\_\rangle_\mathcal{C})$, is an $A_\infty$-category $\mathcal{C}$ together with a  symmetric non-degenerate chain map $$\langle\_,\_\rangle_{\mu\lambda}:\  Hom_\mathcal{C}(\mu,\lambda)\otimes Hom_\mathcal{C}(\lambda,\mu)\rightarrow k[-d]$$
    for all $\mu,\lambda\in Ob(\mathcal{C})$ which satisfies a certain compatibility with the $A_\infty$-compositions, see eg. equation 3.1.29 of \cite{Ul25a}.
    \end{df}
Any cyclic $A_\infty$-category is canonically a strong proper Calabi-Yau category and thus a weak one, see around corollary 2.26 of \cite{AmTu22}. A weak proper Calabi-Yau category together with a splitting of the non-commutative Hodge filtration (a central ingredient in this paper, which we introduce later in definition \ref{splitting}) naturally induce a strong proper Calabi-Yau structure, see lemma 3.16 of \cite{AmTu22}. Lastly in characteristic zero one can strictify a strong proper Calabi-Yau category to a cyclic one, see \cite{KoSo24} theorem 10.7. 

Thus we may pretend that all these notions are roughly equivalent for our purposes. However, the notion of cyclic $A_\infty$-category is not preserved under general $A_\infty$-functors (whereas the one of proper Calabi-Yau categories is).

In fact part of our objects depend only on a small part of the datum of a cyclic $A_\infty$-category.
 \begin{df}\label{ColV}
    We say $$V_B:=(V_{ij},\langle\_,\_\rangle_{ij})_{i,j\in B}$$ is a \emph{collection of cyclic chain complexes of dimension $d$} if it is a cyclic $A_\infty$-category $\mathcal{C}$ of dimension $d$ with zero (higher) compositions and $B=Ob(\mathcal{C})$. 
\end{df}
That is, in particular any cyclic $A_\infty$-category gives rise to a collection of cyclic chain complexes by forgetting the (higher) compositions.

\subsection{Splitting of non-commutative Hodge filtration}
\begin{df}\label{splitting}
    Let $\mathcal{C}$ be an $A_\infty$-category. A \emph{splitting of the non-commutative Hodge filtration} is a chain map
$$s:(CH_*(\mathcal{C}),d_{Hoch})\rightarrow (CH_*(\mathcal{C})\llbracket u\rrbracket,d_{Hoch}+uB)$$
splitting the canonical projection 
$$\pi:(CH_*(\mathcal{C})\llbracket u\rrbracket),d_{Hoch}+uB)\rightarrow(CH_*(\mathcal{C}),d_{Hoch}),$$
that we require $s\circ\pi=id$.
\end{df}
\begin{remark}
Let $\mathcal{C}$ be a $\mathbb{Z}$-graded smooth and proper $A_\infty$-category. By \cite{kal17} the non-commutative Hodge to de-Rham spectral sequence degenerates and converges to negative cyclic homology, as explained eg. in the proof of theorem 5.45 of \cite{She19}. Because we are working over a field of characteristic zero this implies that there is a (non-canonical) identification of negative cyclic homology
\begin{equation}\label{noncan}
HC_*^-(\mathcal{C})\cong HH_*(\mathcal{C})\llbracket u\rrbracket.\end{equation}
\end{remark}
A choice of the splitting of the non-commutative Hodge filtration provided in particular the choice of such an isomorphism. Indeed, extending a splitting $s$ $u$-linearly gives a map 
$$(CH_*(\mathcal{C})\llbracket u\rrbracket,d_{Hoch}) \rightarrow (CH_*(\mathcal{C})\llbracket u\rrbracket,d_{Hoch}+uB),$$
which we can think of as a formal power series of endomorphisms of Hochschild chains whose first term is invertible (in fact it is the identity, by the splitting condition), thus the whole formal power series is invertible. Thus this map induces an isomorphism \eqref{noncan}.
\subsection{$L_\infty$-Algebras}
We recall the following notions, basically citing the succinct presentation of \cite{wi11}:

Let $V$ be a graded vector space. We denote the symmetric algebra by $Sym V=\bigoplus_{n\geq 0} V^{\otimes n} / I$, where $I$ is the two-sided ideal generated by relations $x\otimes y -(-1)^{|x||y|}y\otimes x$. The product will be denoted by $\odot$. For example, the expressions $x_1\odot \cdots \odot x_n :=[x_1\otimes \cdots \otimes x_n]$ generate $Sym^nV$ as a vector space. Let $Sym^+ V:= \bigoplus_{n\geq 1} Sym^n V$, with grading given by $|x_1\odot \cdots \odot x_n|=\sum_j|x_j|$. This space carries the structure of a graded cocommutative coalgebra without counit, with comultiplication given by
\begin{equation}\label{coal_def}
\Delta( x_1 \odot \cdots \odot x_n) = \sum_{\substack{I\sqcup J=[n] \\ |I|,|J|\geq 1}} \epsilon(I,J) \bigodot_{i\in I} x_i \otimes \bigodot_{j\in J} x_j.
\end{equation}
Here $\epsilon(I,J)$ is the sign of the ``shuffle'' permutation bringing the elements of $I$ and $J$ corresponding to \emph{odd} $x$'s into increasing order. Note that $\epsilon(I,J)$ implicitly depends on the degrees of the $x$'s.

\begin{df}
Let $(\mathcal{C},\Delta)$, $(\mathcal{C}',\Delta')$ be graded coalgebras. A linear map $\mathcal{F}\in Hom_k(\mathcal{C},\mathcal{C}')$ is called \emph{degree $k$ morphism of coalgebras} if $\Delta' \circ \mathcal{F} = (\mathcal{F}\otimes \mathcal{F})\circ \Delta$. A linear map $Q\in Hom_k(\mathcal{C},\mathcal{C})$ is called a \emph{degree $k$ coderivation} on $\mathcal{C}$ if  $\Delta \circ Q = (Q\otimes 1 + 1\otimes Q)\circ \Delta$. Here we use the Koszul sign rule, e.g., $(1\otimes Q)(x\otimes y) = (-1)^{k|x|}x\otimes Qy$ etc.
\end{df}

\begin{remark}\label{nütz}
Any coderivation $Q$ on $Sym^+V$ (coalgebra morphism $\mathcal{F}:Sym^+V\to Sym^+W$) is uniquely determined by its composition with the projection $Sym^+V\to Sym^1V=V$, respectively (${Sym^+W\to Sym^1W=W}$). The restriction to $Sym^nV$ of this composition will be denoted by ${Q_n\in Hom(Sym^kV,V)}$ ($\mathcal{F}_n\in Hom(Sym^kV,W)$) and called the $n$-th ``Taylor coefficient'' of $Q$ ($\mathcal{F}$).
\end{remark}
\begin{df}\label{coalgebra}
An $L_\infty$-algebra structure on a graded vector space $\mathfrak{g}^\bullet$ is a degree 1 coderivation $Q$ on $Sym^+(\mathfrak{g}[1])$ such that $Q^2=0$. A morphism of $L_\infty$ algebras $${F}:(\mathfrak{g},Q)\rightsquigarrow (\mathfrak{g}',Q')$$ is a degree 0 coalgebra morphism $C({F}):S^+(\mathfrak{g}[1])\to Sym^+((\mathfrak{g}')[1])$ such that $C({F})\circ Q=Q'\circ C({F})$.
\end{df}

In components, the $L_\infty$-relations read
\[
\sum_{\substack{I\sqcup J=[n] \\ |I|,|J|\geq 1} } \epsilon(I,J) Q_{|J|+1}(Q_{|I|}(\bigodot_{i\in I} x_i) \odot \bigodot_{j\in J} x_j) = 0.
\]
\begin{ex}\label{dgla}
Let $(\mathfrak{g},d,\{\_,\_\})$ be a differential graded Lie algebra. Then the assignments $Q_1(x)=d x$, $Q_2(x,y)=(-1)^{|x|}\{x,y\}$, $Q_n=0$ for $n=3,4,..$ define an $L_\infty$-algebra structure on $\mathfrak{g}$. 
Here and everywhere in the paper $|x|$ is the degree wrt. the grading on the coalgebra.
\end{ex}
\begin{ex}\label{wmg}
An $L_\infty$-morphism $\mathcal{F}$ between dglas $\mathfrak{g}$, $\mathfrak{g}'$ has to satisfy the relations
\begin{multline}
\label{equ:dglaLinfty}
Q_1'\mathcal{F}_n(x_1\odot \cdots \odot x_n) + \frac{1}{2} \sum_{\substack{I\sqcup J=[n] \\ |I|,|J|\geq 1} }\epsilon(I,J)
Q_2'(\mathcal{F}_{|I|}(\bigodot_{i\in I} x_i) \odot \mathcal{F}_{|J|}( \bigodot_{j\in J} x_j) )
= \\ =
\sum_{i=1}^n \epsilon(i,1,\dots,\hat{i},\dots, n) \mathcal{F}_n( Q_1(x_i)\odot x_1\odot \cdots\odot\hat{x}_i \odot \cdots \odot x_n)
+ \\ +
\frac{1}{2}
\sum_{i\neq j}^n \epsilon (i,j,\dots,\hat{i},\dots,\hat{j},\dots, n) \mathcal{F}_{n-1}( Q_2(x_i\odot x_j)\odot x_1\odot\cdots\odot \hat{x}_i \odot \cdots \odot\hat{x}_j \odot \cdots x_n).
\end{multline}
Here the factor $\epsilon(i,j,1,..,\hat{i},..,\hat{j},.., n)$ is the sign of the permutation on the \emph{odd} $x$'s that brings $x_i$ and $x_j$ to the left and analogously for $\epsilon(i,1,\dots,\hat{i},\dots, n)$.

\end{ex}

\section{String Field Theory BD Algebras}
\begin{df}\label{cBD}
Let $\mathcal{C}$ be a cyclic $A_\infty$-category of degree $d$. The \emph{(free) quantum observables of the closed string field theory associated to $\mathcal{C}$} 
$$\mathcal{F}^{c}(\mathcal{C}):=\Big(Sym\Big(CH_*(\mathcal{C})[u^{-1}][d-2]\Big)^{-*}\llbracket\gamma\rrbracket,d_{Hoch}+uB+\gamma\Delta_c,\{\_,\_\}_c\Big),$$
 is the $(2d-5)$-twisted Beilinson-Drinfeld over $k\llbracket\gamma\rrbracket$, that is $d_{Hoch}+uB+\gamma\Delta_c$ is a differential and $\{\_,\_\}_c$ is a $(2d-5)$ shifted Poisson bracket, which together with the natural multiplication satisfy the BD relation \ref{BDrel}. More precisely the algebraic operations are defined exactly as in \cite{AmTu22}, page 35, that is  $\mathcal{F}^{c}(\mathcal{C}):=\mathfrak{h}_\mathcal{C}$, just that we carry the shifts\footnote{Further the notation $-*$ means that we reverse the grading of what is inside the brackets, which is why eg. $|d_{Hoch}|=1.$} with us (and we further ignore the variable $\lambda$ from there). See also section 4.1 of \cite{CaTu24} for a more detailed description of the algebraic operations in the case of a category with one object. 
\end{df}

\begin{remark}
This algebraic structure was originally described in \cite{Cos05}, but with slightly different grading conventions. Here we chose to follow \cite{CaTu24}, but had to perform yet different shifts to be compatible with the open sector.
\end{remark}
\begin{remark}
    In \cite{CL15} the authors describe how for $\mathcal{C}$ the derived category of coherent sheaves on $\mathbb{C}^3$ the BD algebra $\mathcal{F}^{c}(\mathcal{C})$ describes the free quantum observables of BCOV theory, that is the closed string field theory of the B-model.
\end{remark}
\begin{df}\label{cBD_t}
   Let $\mathcal{C}$ be a cyclic $A_\infty$-category of degree $d$. We denote by $\mathcal{F}^{c}(\mathcal{C})^{triv}$ \emph{the $(2d-5)$-twisted BD algebra which has the same underlying graded algebra as $\mathcal{F}^{c}(\mathcal{C}),$ but with zero bracket and whose differential is just induced from the derivation $d_{Hoch}+uB$}.
\end{df}
\begin{df}\label{cBD_Tr}
   Let $\mathcal{C}$ be a cyclic $A_\infty$-category of degree $d$. We denote by $\mathcal{F}^{c}(\mathcal{C})^{Triv}$ \emph{the $(2d-5)$-twisted BD algebra which has the same underlying graded algebra as $\mathcal{F}^{c}(\mathcal{C}),$ but with zero bracket and whose differential is just induced from the derivation $d_{Hoch}$}.
\end{df}

\begin{df}\label{oBD}
Let $V_B$ be a collection of cyclic chain complexes of degree $d$, recalling definition \ref{ColV}. Then we denote the \emph{large $N$ observables of open string field theory of a stack of $N$ branes of type $\Lambda$} by $$\mathcal{F}^o(V_B):=\Big(Sym\big(Cyc^*(V_B)[d-4]\big) \llbracket\gamma\rrbracket,d+\nabla+\gamma\delta,\{\_,\_\}_o\Big).$$ 
It has a $(2d-5)$-twisted Beilinson-Drinfeld algebra structure over $k\llbracket\gamma\rrbracket$, a formal variable of degree $(6-2d)$. This follows by identifying $\mathcal{F}^o(V_B)=\mathcal{F}^{pq}(V_B)[d-3],$ see around theorem 3.1.35 of \cite{Ul25a} for notations and details on the algebraic structures. Roughly its underlying graded vector space is  built as the symmetric algebra on composable cyclic words (including the empty one) in the dual of the Hom spaces of $V_B$, and all adjoined $\gamma$, a formal variable of degree $(6-2d)$. 

If $V_B$ is induced from a full subcategory with objects $\Lambda$ of a given cyclic $A_\infty$-category $\mathcal{C}$, recalling remarks around \ref{ColV}, we denote $\mathcal{F}^o(V_B)=:\mathcal{F}^o(\Lambda).$
\end{df}
\begin{remark}
    As written we shifted the initial $(d-2)$-twisted BD algebra from \cite{Ul25a}, theorem 3.1.35 by $(d-3)$. We made that choice there to be more compatible with the target of the LQT map, which naturally has a $(d-2)$-shifted BD algebra structure, whereas here it seems more natural not to perform that shift. We hope no confusion arises.  
\end{remark}
\begin{remark}
    The BD structure is induced from the well-studied (notably in string topology) involutive bi Lie algebra structure on cyclic cochains \cite{NaWi19, CiFuLa20, Ca16}. Note that $\mathcal{F}^o(\Lambda)$ is also the doamin of the quantized LQT theorem \cite{GGHZ21, Ul25a}, which for the derived category of coherent sheaves on $\mathbb{C}^3$ exhibits the BD algebra $\mathcal{F}^o(\Lambda)$ as the free quantum observables of large $N$ holomorphic Chern-Simons theory, which is the open string field theory of the B-model on a stack of $N$ space-filling branes \cite{CL15}.
\end{remark}

\section{Recollection on Circle-Action Trivialization of Closed SFT}
Note that by forgetting the multiplication a $(2d-5)$-twisted BD algebra becomes a dg $(2d-5)$-shifted Lie algebra, which we denote by the same letter by abuse of notation. We obtain dg Lie algebras by shifting by $-(2d-5)$ a dg $(2d-5)$ shifted Lie algebra.

In this section we recall from \cite{CaTu24} the construction of an $L_\infty$-quasi-isomorphism 
\begin{equation}\label{inLqi}
\mathcal{K}_s:\ \mathcal{F}^{c}(\mathcal{C})[5-2d]\rightsquigarrow\mathcal{F}^{c}(\mathcal{C})^{triv}[5-2d],\end{equation}
given a splitting $s$ of the non-comutative Hodge filtration (definition \ref{splitting}). Note, however, that we have introduced a different grading; we carefully keep track of the $\mathbb{Z}$-grading.  The Taylor coefficients of $\mathcal{K}_s$ (compare remark \ref{nütz}) are given by maps
\begin{equation}\label{tayf}
\mathcal{K}_s^m:Sym^m(\mathcal{F}^{c}(\mathcal{C})[6-2d])\rightarrow \mathcal{F}^{c}(\mathcal{C})^{triv}[6-2d],\end{equation}
which uniquely define the morphism \eqref{inLqi}. Since this morphism constitutes one of the main ingredient of this note, we recall its definition here, after \cite{CaTu24}, whose notation we follow. First we recall some
notions of graphs, introduced in \cite{geKa96}. A labeled
graph means for us a graph $G$ (possibly with leaves) endowed with a
`loop defect' labeling function $g: V_G\rightarrow \mathbb{Z}_{\geq 0}$ on the set of its
vertices $V_G$. 
We use the following notations for a labeled graph $G$:
\begin{itemize}
\item[--] $E_G$ denotes the set of edges, the cardinality of that set denoted $|e|$;
\item[--] $h_v$ denotes the set of halfedges belonging to a vertex $v$;
\item[--] The valency of a vertex $v\in V_G$ is denoted by $|h_v|$, the cardinality of halfedges belonging to that vertex;
\item[--] the number of vertices of a graph is denoted by $|v|$.
\end{itemize}
The betti number of a labeled graph is defined to be
\begin{equation}\label{betti}
 g(G):= \sum_{v\in V_G} g(v)+{rank\,} H_1(G), \end{equation}
and we recall that for connected graph we have \begin{equation}
\label{rank}{rank\,} H_1(G)=|e|-|v|+1.\end{equation}
\begin{itemize}
\item[--] If $G$ has $m$ vertices, a marking of $G$ is a bijection
  \[ f: \{1,\cdots,m\} \rightarrow V_G. \] 
 \item[--] An isomorphism between two marked
and labeled graphs is an isomorphism of the underlying labeled graphs
that also preserves the markings.
\end{itemize}

We denote by $\widetilde{\Gamma(g,n)}_m$  the set of isomorphism classes of connected {\em marked}
graphs of betti number $g$ with $n$ leaves and with $m$ vertices. 

Next we recall that 
$$\mathcal{F}^{c}(\mathcal{C})[6-2d]=Sym\Big(CH_*(\mathcal{C})[u^{-1}][d-2]\Big)^{-*}\llbracket\gamma\rrbracket[6-2d]$$
and fix the notation $$\mathcal{H}:=(CH_*(\mathcal{C})[u^{-1}][d-2])^{-*}$$ so that 
$$\mathcal{F}^{c}(\mathcal{C})[6-2d]=Sym(\mathcal{H})\llbracket\gamma\rrbracket[6-2d].$$
Further for an element $y=:x_1x_2\cdots x_j \in Sym^{j}(\mathcal{H})$ let
$\widetilde{y} \in \mathcal{H}^{\otimes j}$ be its
desymmetrization,
\begin{equation}\label{desym}
 \widetilde{y}:=\sum_{\sigma\in \Sigma_j} \epsilon \cdot
x_{\sigma(1)}\otimes\cdots\otimes x_{\sigma(j)} \in \mathcal{H}^{\otimes j}.
\end{equation}
Here $\epsilon$ is the Koszul sign for permuting the elements
$x_1,\ldots, x_i$ with respect to the degree in $\mathcal{H}$. 
Given $n>0,g\geq0$, we denote by 
$$(\_)|_{g,n}:\mathcal{F}^{c}(\mathcal{C})[6-2d]\rightarrow (\mathcal{F}^{c}(\mathcal{C})[6-2d])|_{g,n}$$
the projection on the sub vector space given by symmetric words of length $n$ and of $\gamma$-coefficient~$g$.

Given a splitting $s$ from definition \ref{splitting} we can construct a symmetric bilinear form
  \begin{equation}\label{Hsym}
H^{sym}_s: \mathcal{H}^{\otimes 2} \rightarrow k[6-2d], \end{equation}  see page 38 of \cite{AmTu22} or in more detail proposition 7.5 of \cite{CaTu24} for the definition - carried over to our shifting and grading convention. 

Using these notations we can recall the definition of the degree zero linear maps \eqref{tayf}
$$\mathcal{K}_s^m: Sym(\mathcal{F}^{c}(\mathcal{C})[6-2d])\rightarrow \mathcal{F}^{c}(\mathcal{C})^{triv}[6-2d],$$
defined for each $m\geq 1$. We have
\[ \mathcal{K}_s^m := \sum_{g,n} \sum_{ (G,f)\in \widetilde{\Gamma(g,n)}_m}
  \frac{1}{|Aut(G,f)|}\cdot K_{(G,f)}, \]
and we describe the map $K_{(G,f)}$: Let $(G,f)\in \widetilde{\Gamma(g,n)}_m$ be a marked graph. Set
$g_i = g(f(i))$, and $n_i = |h_{f(i)}|$.  The map $K_{(G,f)}$ has only non-trivial components
\begin{equation}\label{betco}
\bigotimes_{i=1}^m  \mathcal{F}^{c}(\mathcal{C})[6-2d]|_{g_i,n_i}  \rightarrow \mathcal{F}^{c}(\mathcal{C})[6-2d]|_{g,n},
\end{equation}
which leaves us to specify for $y_i\in Sym^{n_i}(\mathcal{H})[6-2d]\ (i=1,\ldots,m)$ the expression $$K_{(G,f)} (y_1\gamma^{g_1},\ldots, y_m\gamma^{g_m})=Z\gamma^g.$$ 
\begin{ctr}\label{Feynm}
Here the element
$Z$ is computed by the following Feynman-type procedure. 
\begin{enumerate}
\item Given $y_i\in Sym^{n_i}(\mathcal{H})[6-2d]$ for $i=1,\ldots,m$ 
decorate the half-edges adjacent to each vertex $v_i$ by
$\widetilde{y}_i$, recalling the desymmetrization operation \eqref{desym}. Tensoring together the results over all the
vertices yields a tensor of the form
\begin{equation}\label{abt}
 \bigotimes_{i=1}^m \widetilde{\gamma_{i}}\in \mathcal{H}^{\otimes (\sum
    n_i)}[|v(G)|(6-2d)].  \end{equation}
The order in which these elements are tensored is the one
given by the marking $f$.

\item  For each internal edge $e$ of $G$ contract the corresponding
components of the above tensor \eqref{abt} using the symmetric bilinear form $H^{sym}_s$, built from the splitting $s$. When applying the
contraction we always permute the tensors to bring the two terms
corresponding to the two half-edges to the front, and then apply the
contraction map. The ordering of the set $E_G$ does not matter since
the operator $H^{sym}_s$ is even; also the ordering of the two
half-edges of each edge does not matter since the operator $H^{sym}_s$
is (graded) symmetric.
\item Read off the remaining tensor components (corresponding to the
leaves of the graph) in any order, and regard the result as the
element $$Z\in Sym^n (\mathcal{H})[6-2d]$$ via the canonical projection map
$\mathcal{H}^{\otimes n}\rightarrow Sym^n \mathcal{H}$, followed by a shift by $(|v|-1)(6-2d)$. 
\end{enumerate}
Tracking the degrees it follows that $|Z|=\sum_i|y_i|+e(2d-6)+(1-v)(2d-6).$ This shows that $|K_{(G,f)}|=0$, since we have
$$|Z\gamma^g|=|Z|+g(6-2d)=(g-2d)\sum_ig_i+\sum_i|y_i|,$$
recalling the formula \eqref{betti} for the betti number $g$ of a graph.

\end{ctr}

\section{Circle-Action Trivialization of Open-Closed SFT}
In this section we prove the main theorem A. Let us fix a dimension $d$ cyclic $A_\infty$-category $\mathcal{C}$, a splitting $s$ of the non-commutative Hodge filtration (see definition \ref{splitting}) and an auxiliary $(2d-5)$-twisted BD algebra $W$ (which we will later take to be $\mathcal{F}^{c}(\Lambda)$).

Recall (from lemma \ref{tensorBD}) that we can consider the tensor product $(2d-5)$-twisted BD algebras $\mathcal{F}^{c}(\mathcal{C})\otimes W$ and $\mathcal{F}^{c}(\mathcal{C})^{triv}\otimes W$, again recalling definitions \ref{cBD} and \ref{cBD_t}.
By forgetting their multiplications and shifting we find Lie algebras $\mathcal{F}^{c}(\mathcal{C})\otimes W[5-2d]$ and  $\mathcal{F}^{c}(\mathcal{C})^{triv}\otimes W[5-2d]$.
The goal of this section is to construct an $L_\infty$ quasi-isomorphism denoted 
\begin{equation}\label{L_inf}
\mathcal{K}_s\otimes m:\ \mathcal{F}^{c}(\mathcal{C})\otimes W[5-2d]\rightsquigarrow \mathcal{F}^{c}(\mathcal{C})^{triv}\otimes W[5-2d].\end{equation}

Equivalently (see definition \ref{coalgebra}) such a map \eqref{L_inf} can be described as a map of dg coalgebras 
\begin{equation}\label{CoALG}
C(\mathcal{K}_s\otimes m):\ Sym\big(\mathcal{F}^{c}(\mathcal{C})\otimes W[6-2d],d_{oc}\big)\rightarrow Sym\big(\mathcal{F}^{c}(\mathcal{C})^{triv}\otimes W[6-2d],d_{oc}^t\big),\end{equation}
where the left and right hand side have the standard coalgebra structure (compare definition \ref{coal_def}). The differentials $d_{oc}$ and $d_{oc}^t$ of degree 1 are encoding the differential graded Lie algebra structures of $\mathcal{F}^{c}(\mathcal{C})\otimes W[5-2d]$ respectively $\mathcal{F}^{c}(\mathcal{C})^{triv}\otimes W[5-2d]$, see example \ref{dgla}. 

We define the underliyng map of graded coalgebras of \eqref{CoALG}:
\begin{df}
Let
\begin{equation}\label{cogb}
C(\mathcal{K}_s\otimes m):\ Sym\big(\mathcal{F}^{c}(\mathcal{C})\otimes W[6-2d]\big)\rightarrow Sym\big(\mathcal{F}^{c}(\mathcal{C})^{triv}\otimes W[6-2d]\big)\end{equation}
be the map given by
\begin{equation*}
C(\mathcal{K}_s\otimes m)^n((x_1\otimes y_1)\odot \ldots\odot(x_n\otimes y_n)):=(-1)^*\mathcal{K}_s^n(x_1\odot\ldots\odot x_n)\otimes (y_1\cdot\ldots\cdot y_n),\end{equation*}
where the sign arises from the Koszul rule of permuting all the $x$'s to the left and the $y$'s to the right. 
\end{df}
\begin{remark}
As explained in remark \ref{nütz} a coalgebra morphism on such a coalgebra is uniquely determined by its Taylor coefficients. In the formula above those are denoted by $\mathcal{K}_s^n$.
\end{remark}
It remains to verify that the map \eqref{cogb} commutes with the differentials, ie. defines a map of $L_\infty$-algebras and thus determines the required map \eqref{L_inf}. The fact that it is an $L_\infty$ \emph{quasi-isomorphism} follows then directly since its first Taylor component will be given by the tensor product of a quasi-isomorphism (by \eqref{inLqi}) and the identity, thus being a quasi-isomorphism. 
\begin{theorem}\label{cwd}
Given a dimension $d$ cyclic $A_\infty$-category $\mathcal{C}$, a splitting $s$ of the non-commutative Hodge filtration (see definition \ref{splitting}) and a $(2d-5)$-twisted BD algebra $W$. Then we have that 
    $$C(\mathcal{K}_s\otimes m)\circ d_{c}=d_{c}^t\circ C(\mathcal{K}_s\otimes m)$$
    for the map $C(\mathcal{K}_s\otimes m)$ from \eqref{cogb}. That is, it defines the desired $L_\infty$ quasi-morphism
$$\mathcal{K}_s\otimes m:\ \mathcal{F}^{c}(\mathcal{C})\otimes W[5-2d]\rightsquigarrow \mathcal{F}^{c}(\mathcal{C})^{triv}\otimes W[5-2d].$$
\end{theorem}
The proof of this theorem occupies the rest of this note (see summary \ref{Zusa}):

As explained in example \ref{wmg}, since we are actually dealing with dg Lie algebras, proving thereom \ref{cwd} is equivalent to showing equation \eqref{equ:dglaLinfty}. Translating our notation to that of equation \eqref{equ:dglaLinfty}, further recalling example \ref{dgla} and equation \eqref{tenbra} means that  
\begin{equation*}
Q_1'=d_{c,t}+d_w,\ \ \ Q_1=d_{c}+d_w,
\end{equation*}
\begin{equation*}
 Q_2'(x_1\otimes y_1,x_2\otimes y_2)=(-1)^{|x_1|+|y_1|+|x_2|}m_c(x_1,x_2)\otimes\{y_1,y_2\}_w,
 \end{equation*}
 \begin{equation*}
    Q_2(x_1\otimes y_1,x_2\otimes y_2)=(-1)^{|x_1|}\{x_1,x_2\}_c\otimes m_w(y_1,y_2)+(-1)^{|x_1|+|y_1|+|x_2|}m_c(x_1,x_2)\otimes\{y_1,y_2\}_w,
    \end{equation*}
    
where we denote (momentarily) by $m_c$ respectively $m_w$ the commutative product on $\mathcal{F}^{c}(\mathcal{C})$ (and by abuse of notation on $\mathcal{F}^{c}(\mathcal{C})^{triv}$) respectively on $W$, by $d_w$ the BD differential on $W$; by $d_c$ the BD differential of $\mathcal{F}^{c}(\mathcal{C})$ and by $d_{c,t}$ the BD differential of $\mathcal{F}^{c}(\mathcal{C})^{triv}$. 

Summarizing, proving theorem \ref{cwd} is equivalent to showing that following equation\footnote{Here $(-1)^{*_1}$ is given by the Koszul rule of permuting the $x$'s to the left and the $y$'s to the right. $(-1)^{*_2}=(-1)^{\sum_{i}|x_i|+\sum_{i\in L_1}|y_i|}$ times the Koszul rule of permuting the factors of ${(x_1\otimes y_1)\odot\cdots \odot (x_m\otimes y_m)}$ into the respective order. $(-1)^{*_3}=\epsilon(i,j,1,..,\hat{i},..,\hat{j},.., n)$. $(-1)^{*_4}=(-1)^{(\sum_i|x_i|)+|y_i|} $ times the Koszul rule of permuting the factors. $(-1)^{*_5}=(-1)^{|x_i|}$ times the Koszul rule of permuting the factors.} holds:
\begin{align}\label{inter}
&(d_{c,t}+d_w)(-1)^{*_1}\mathcal{K}_n(x_1\odot\cdots \odot x_n)\otimes (y_1\cdot\ldots\cdot y_m)\nonumber\\
+&\sum_{\substack{L_1\sqcup L_2=[n]\nonumber \\ |L_1|,|L_2|\geq 1}}(-1)^{*_2}\big(\mathcal{K}_{|L_1|}(\bigodot_{i\in L_1} x_i)\cdot\mathcal{K}_{|L_2|}(\bigodot_{i\in L_2} x_i)\big)\otimes \{\Pi_{i\in L_1} y_i,\Pi_{j\in L_2} y_j\}_w\nonumber\\
=&\sum_{i}(-1)^{*_3}(\mathcal{K}_s\otimes m)^n\big((d_{c}+d_w)(x_i\otimes y_i)\odot(x_1\otimes y_1)\odot\cdots \widehat{(x_i\otimes y_i)}\cdots \odot (x_m\otimes y_n)\big)\nonumber\\
+&\sum_{i\neq j}(-1)^{*_4}\mathcal{K}_{m-1}\big((x_i\cdot x_j)\odot x_1\odot\ldots \hat{x_i}\ldots\hat{x_j}\ldots\odot x_m\big)\otimes\big(\{y_i,y_j\}_w\cdot  y_1\cdot\ldots \hat{y_i}\ldots\hat{y_j}\ldots\cdot y_m\big)\nonumber\\
+&\sum_{i\neq j}(-1)^{*_5}\mathcal{K}_{m-1}\big(\{x_i, x_j\}_c\odot x_1\odot\ldots \hat{x_i}\ldots\hat{x_j}\ldots\odot x_m\big)\otimes  \big((y_i\cdot y_j)\cdot y_1\cdot\ldots\hat{y_i}\ldots\hat{y_j}\ldots \cdot y_m\big)
\end{align}

Using following two facts \eqref{BDr} and \eqref{KiL} we show that proving theorem \ref{cwd}, which is equivalent to proving that equation \eqref{inter} holds, is in fact equivalent to proving lemma \ref{key} below.

The fact that $W$ is a BD algebra implies that\footnote{Here the sign $(-1)^{*_6}$ results from permuting the odd operator $d_w$ through the $y_1,\ldots,y_{i-1}.$ Further we have $(-1)^{*_7}=(-1)^{|y_i|}\epsilon (i,j,\dots,\hat{i},\dots,\hat{j},\dots, n) $}
\begin{align}\label{BDr}
\mathcal{K}_n(x_1\odot\cdots \odot x_n)\otimes d_w(y_1\cdot\ldots\cdot y_n)&\nonumber\\
=\sum_i(-1)^{*_6}\mathcal{K}_n(x_1\odot\cdots \odot x_n)\otimes (y_1\cdot&\ldots \cdot d_wy_i\cdot\ldots\cdot y_m)\\+\sum_{i\neq j}(-1)^{*_7}&\mathcal{K}_n(x_1\odot\cdots \odot x_n)\otimes\big(\gamma\{y_i,y_j\}_w\cdot  y_1\cdot\ldots \hat{y_i}\ldots\hat{y_j}\ldots\cdot y_m\big).\nonumber
\end{align}
Further by equation \eqref{equ:dglaLinfty} the fact that the maps $\mathcal{K}^n_s$ define an $L_\infty$-algebra morphism (by equation 26 of \cite{AmTu22}, following theorem 7.1 of \cite{CaTu24}) implies that\footnote{Here $(-1)^{*_8}=\epsilon(i,1,\dots,\hat{i},\dots, n)$ and $(-1)^{*_9}=\epsilon(i,j,1,..,\hat{i},..,\hat{j},.., n)(-1)^{|x_i|}$}
\begin{align}\label{KiL}
    d_{c,t}\mathcal{K}_n(x_1\odot\cdots \odot& x_n)\otimes  \big(y_1\cdot\ldots \cdot y_m\big)\nonumber\\
=  \sum_{i}(-1)^{*_8}\mathcal{K}_s^n\big(d_{c}x_i\odot x_1&\odot\cdots \widehat{x_i}\odot \cdots \odot x_m\big)\otimes  \big(y_1\cdot\ldots \cdot y_m\big)\\ 
&+\sum_{i\neq j}(-1)^{*_9}\mathcal{K}_{m-1}\big(\{x_i, x_j\}_c\odot x_1\odot\ldots \hat{x_i}\ldots\hat{x_j}\ldots\odot x_m\big)\otimes  \big(y_1\cdot\ldots \cdot y_m\big)\nonumber.
\end{align}
Using these facts, showing that equation \eqref{inter} reduces to the equation in the following lemma is now basic algebra: Multiplying equality \eqref{KiL} by the Koszul sign of permuting the $x$'s of ${(x_1\otimes y_1)\odot\cdots \odot (x_m\otimes y_m)}$ to the left and the $y$'s to the right shows that the first summand of the first line of \eqref{inter} is equal to the first summand of the third summand and the last line. Multiplying equality \eqref{BDr} by the same Koszul sign times $(-1)^{\sum_i|x_i|}$ implies that the second summand of the first line of \eqref{inter} is equal to the second summand of the third line of \eqref{inter} plus the last line of \eqref{BDr} (multiplied by that same sign). Since also the second and forth line of \eqref{inter} is multiplied by $(-1)^{\sum_i|x_i|}$ we have proven the claim. Thus to prove theorem \ref{cwd} it remains to show the following.
\begin{lemma}\label{key}
For all $m>1$ and $(x_1\otimes y_1)\odot\cdots \odot (x_m\otimes y_m)\in Sym^m(\mathcal{F}^{c}(\mathcal{C})\otimes W[6-2d])$ we have 
\begin{align*}
\sum_{i\neq j}(-1)^{*_a}\mathcal{K}_{m-1}\big(&(x_i\cdot x_j)\odot x_1\odot\ldots \hat{x_i}\ldots\hat{x_j}\ldots\odot x_m\big)\otimes\big(\{y_i,y_j\}_w\cdot  y_1\cdot\ldots \hat{y_i}\ldots\hat{y_j}\ldots\cdot y_m\big)\\
=\sum_{\substack{L_1\sqcup L_2=[m] \\ |L_1|,|L_2|\geq 1}}(-1)^{*_b}&\big(\mathcal{K}_{|L_1|}(\bigodot_{i\in L_1} x_i)\cdot\mathcal{K}_{|L_2|}(\bigodot_{i\in L_2} x_i)\big)\otimes \{\Pi_{i\in L_1} y_i,\Pi_{j\in L_2} y_j\}_w\\
+&\gamma\sum_{i\neq j}(-1)^{*_c}\mathcal{K}_m(x_1\odot\cdots \odot x_m)\otimes\big(\{y_i,y_j\}_w\cdot y_1\cdot\ldots \hat{y_i}\ldots\hat{y_j}\ldots\cdot y_m\big)
\end{align*}
The signs arise from the Koszul rule of permuting the (odd) factors of ${(x_1\otimes y_1)\odot\cdots \odot (x_m\otimes y_m)}$ into the respective order; further from moving the $y_i$ respectively the $\Pi_{i\in L_1} y_i$ into the odd bracket $\{\_,\_\}_w.$ 
\end{lemma}
In order to prove this key lemma, which is found below the proof of lemma \ref{easy}, we need to elaborate a bit: 
\subsection{An Interlude on Graphs}
In this subsection we prove theorem \ref{gt}, which is essential in the proof of the key lemma \ref{key}.
\begin{cons}\label{Cons1}
Given a connected marked graph $G$ with $m-1$ vertices 
and a partition $I\sqcup J=h_1$ of the halfedges $h_1$ belonging to the (for example) first vertex of the graph we can construct a new marked graph with $m$ vertices
where now only the halfedges $I$ belong to the first vertex and the halfedges $J$ are declared to belong to a new vertex, marked $m$.
We denote the new graph by $G^{I,J}$.
However, we need to be careful as in this process we may create two disconnected marked graph, denoted $G^{I}$ respectively $G^{J}$,
which for instance always happens when our starting graph is a star. See the proof of theorem \ref{gt} below for how we mark these new graphs.
\end{cons}
We will see that we are also able to reverse this procedure. To make this more precise we introduce the following notations.
\begin{df}
For a given marked graphs $(G,f)$ and two partitions of their halfedges belonging to the first vertex $I_1\sqcup J_1=h_1$ respectively $I_2\sqcup J_2=h_1$ we say $$ ( I_1,J_1)\sim ( I_2,J_2)\ \ \text{if}\
 \begin{cases} G^{I_1,J_1}\cong G^{I_2,J_2},& \text{if the resulting graphs are connected} \\
G^{I_1}\cong G^{I_2}\ \text{and}\ G^{J_1}\cong G^{J_2},& \text{otherwise.}
\end{cases}
$$
To be precise the isomorphism have to be from the category of marked graphs.
 Further given a graph $(G,f)\in \widetilde{\Gamma(g,n)}_{m-1}$ and $k_1,k_2\in\mathbb{N}_{>0}$ denote by $$P_{(G,f)}^{k_1,k_2}:=\{I\sqcup J= h_{1}\ |\ |I|=k_1,|J|=k_2\}/\sim.$$
\end{df}
\begin{df}\label{grdf}
Fix $0\leq g$, $n<0$ and $k_1,k_2\in\mathbb{N}_{>0}$. Denote
\begin{align*}
A_{k_1,k_2}^{g,n}:=&\{(G,f)\in \widetilde{\Gamma(g,n)}_{m-1}, I,J\in P_{(G,f)}^{k_1,k_2}\},\\
B_{k_1,k_2}^{g,n}:=&\{(G,f)\in \widetilde{\Gamma(g,n)}_m\ |\  |h_{1}|=k_1,|h_{m}|=k_2\}\\
C_{k_1,k_2}^{g,n}:=&\{(G_1,f_1)\in \widetilde{\Gamma(g_1,n_1)}_{m_1},(G_2,f_2)\in \widetilde{\Gamma(g_2,n_2)}_{m_2},\ L_1\sqcup L_2=\{1,\cdots,m\}| 1\in L_1,m\in L_2 \\
&\ g_1+g_2=g, m_1+m_2=m, n_1+n_2=n, 
|h_1(G_1)|=k_1, |h_m(G_2)|=k_2   \}
\end{align*}
Further given $1\leq i<j\leq m$ we denote 
\begin{align*}
\widetilde{B}_{k_1,k_2}^{g,n}:=&\{(G,f)\in \widetilde{\Gamma(g,n)}_m\ |\  |h_{i}|=k_1,|h_{j}|=k_2\}\\
\widetilde{C}_{k_1,k_2}^{g,n}:=&\{(G_1,f_1)\in \widetilde{\Gamma(g_1,n_1)}_{m_1},(G_2,f_2)\in \widetilde{\Gamma(g_2,n_2)}_{m_2},\ L_1\sqcup L_2=\{1,\cdots,m\}| i\in L_1,j\in L_2 \\
&\ g_1+g_2=g, m_1+m_2=m, n_1+n_2=n, 
|h_i(G_1)|=k_1, |h_j(G_2)|=k_2   \}
\end{align*}
\end{df}
Obviously we have that \begin{equation}\label{obv2}
{B}_{k_1,k_2}^{g,n}\cong\widetilde{B}_{k_1,k_2}^{g,n}
\end{equation}
\begin{equation}\label{obv}
{C}_{k_1,k_2}^{g,n}\cong\widetilde{C}_{k_1,k_2}^{g,n}
\end{equation}
by swapping $1$ with $i$ and $m$ with $j$.
\begin{theorem}\label{gt}
   Fix $0\leq g$, $0<n$ and $k_1,k_2\in\mathbb{N}$. Then there is a bijection 
   $$\psi: A_{k_1,k_2}^{g,n}\cong B_{k_1,k_2}^{g,n}\cup C_{k_1,k_2}^{g,n}.$$
\end{theorem}
\begin{proof}
    We have already described in construction \ref{Cons1} how given an element of $A_{k_1,k_2}^{g,n}$ we produce (a) graph(s) on the right hand side. In the connected case we also described the marking. In the disconnected case we simply use the old marking, but reorder according to size. Further in the disconneted case the partition $L_1\sqcup L_2=\{1,\cdots,m\}$ is obtained by remembering which of the previous marked vertices are now part of $G_1$ (including the old first one) and which of the previous marked vertices are now part of $G_2$ (including the new last one, which was obtained by splitting the old first one into two).
    
    Given an element in the right hand side (be it connected or not), whose total number of vertices is $m$, we send it to the graph determined by declaring the first and $m$-th vertex of the initial graph(s) to coincide and which we number by $1$, then.
    In the case the graph came from $B$, i.e. was connected, it is straightforward to number the remaining vertices of the resulting new graph.
    In the case the graph comes from $C$ we use the partition $L_1\sqcup L_2=\{1,\cdots,m\}$ which satisfied the condition that $ 1\in L_1,m\in L_2$ to mark the remaining vertices of the resulting new graph. To define the partition into two sets of the halfedges of the `new' vertex of $\psi^{-1}(G,f)$ we remember which ones came from the first vertex of $(G,f)$ and which came from the $m$-th vertex of $(G,f)$. 
    
    Finally it is straightforward to check that these two assignments are inverse to each other.
\end{proof}
\subsection{The Proof of Theorem A continued}
In turn theorem \ref{gt} allows us to prove following key technical theorem:
\begin{theorem}\label{BVinf}
For $m>1$ and $x_1,\cdots, x_m\in \mathcal{F}^{c}(\mathcal{C})[6-2d]$ and $1\leq i\neq j\leq m$ we have that\footnote{Here $(-1)^*=\epsilon(i,j,1,..,\hat{i},..,\hat{j},.., n)$ and $(-1)^\#=\epsilon(L_1,L_2)$.}
\begin{align*}
&(-1)^*\mathcal{K}_{m-1}\big((x_i\cdot x_j)\odot x_1\odot\ldots \hat{x_i}\ldots\hat{x_j}\ldots\odot x_m\big)\\
=&\gamma\mathcal{K}_m(x_1\odot\cdots \odot x_m)+\sum_{\substack{L_1\sqcup L_2=[m] \\  i\in L_1,j\in L_2}}(-1)^\#\big(\mathcal{K}_{|L_1|}(\bigodot_{i\in L_1} x_i)\cdot\mathcal{K}_{|L_2|}(\bigodot_{i\in L_2} x_i)\big)
\end{align*}
\end{theorem}
\begin{proof}
By definition we have 
\begin{align*}
&(-1)^*\mathcal{K}_{m-1}\big((x_i\cdot x_j)\odot x_1\odot\ldots \hat{x_i}\ldots\hat{x_j}\ldots\odot x_m\big)\\=&\sum_{g,n} \sum_{ (G,f)\in \widetilde{\Gamma(g,n)}_{m-1}}(-1)^*
  \frac{1}{|Aut(G,f)|}\cdot K_{(G,f)}\big((x_i\cdot x_j),x_1,\ldots, \hat{x_i},\ldots,\hat{x_j},\ldots, x_m\big)\\
  =&:E
  \end{align*}

We denote by $\widetilde{ev}_{G,f}(x_1,\cdots,x_m)$ the operation of \textit{contracting exactly once the letters of the symmetric words} $\{x_l\}_{l=1,\cdots,m}\in\mathcal{F}^c(\mathcal{C})[6-2d]$ with each other, using $H^{sym}$ and according to the datum given by a marked graph $(G,f)$ having $m$ vertices, weighted appropriately by $\gamma$. This gives the identification 
\begin{equation}\label{id1}
\widetilde{ev}_{(G,f)}(x_1,\cdots,x_m)=\frac{1}{|Aut(G,f)|}\cdot K_{(G,f)}(x_1,\cdots,x_m),\end{equation}
recalling definition \ref{Feynm}.
We can ask whether we can re-express in a similar way
\begin{equation}\label{summe}
(-1)^*\frac{1}{|Aut(G,f)|}\cdot K_{(G,f)}(x_i\cdot x_j,x_1,\cdots,\hat{x_i},\cdots,\hat{x_j},\cdots,x_m),\end{equation}
where $x_i\in Sym^{k_1}(\cdots)$ and $x_j\in Sym^{k_2}(\cdots)$.
Indeed, again recalling the definition \eqref{Feynm}, this term equals a sum which we can naturally index by $(I,J)\in P_{(G,f)}^{(k_1,k_2)}$ and whose summands we suggestively denote by 
\begin{equation}\label{se}
(-1)^*\widetilde{ev}_{(G,f,I,J)}(x_i, x_j,x_1,\cdots,\hat{x_i},\cdots,\hat{x_j},\cdots,x_m).\end{equation}
\begin{remark}
We will give a precise meaning to \eqref{se} in equations \eqref{ia1} and \eqref{ia2}.
\end{remark}For the moment we repeat that 
\begin{align*}
\sum_{I,J\in P_{(G,f)}}(-1)^*\widetilde{ev}_{(G,f,I,J)}&(x_i, x_j,x_1\cdots,\hat{x_i},\cdots,\hat{x_j},\cdots,x_m)\label{fs}\\
&=(-1)^*\frac{1}{|Aut(G,f)|}\cdot K_{(G,f)}(x_i\cdot x_j,x_1\cdots,\hat{x_i},\cdots,\hat{x_j},\cdots,x_m)\nonumber
\end{align*}
Thus we can write
\begin{equation}\label{E}
E=\sum_{g,n} \sum_{ (G,f)\in \widetilde{\Gamma(g,n)}_{m-1}}\sum_{I,J\in P_{(G,f)}}(-1)^*\widetilde{ev}_{(G,f,I,J)}(x_i, x_j,x_1,\cdots,\hat{x_i},\cdots,\hat{x_j},\cdots,x_m).\end{equation}
\begin{obs}
We note that we can re-express the summands \eqref{se} of \eqref{summe}, applying the rule \ref{Feynm}, as
\begin{align}\label{ia1}
(-1)^*\widetilde{ev}_{(G,f,I,J)}&(x_i,x_j,x_1,\cdots,\hat{x_i},\cdots,\hat{x_j},\cdots,x_m)\nonumber\\=&\gamma\epsilon(i,1,..,\hat{i},..,\hat{j},.., n,j)\widetilde{ev}_{\psi(G,f,I,J)}(x_i,x_1\cdots,\hat{x_i},\cdots,\hat{x_j},\cdots,x_m,x_j)\end{align}
if $\psi(G,f,I,J)\in B_{(k_1,k_2)}^{g,n}$. Here the right hand side is defined in the standard way by rule \ref{id1}.
The additional $\gamma$ factors in \eqref{ia1} arise from the fact that the appearing graphs on the right hand side have the same number of total edges, but one more vertex, thus decreasing the betti number by one as compared to the appearing graphs appearing on the left hand side (recalling definition \ref{betti}).

Similarly we have 
\begin{align}
&(-1)^*\widetilde{ev}_{(G,f,I,J)}(x_i, x_j,x_1,\cdots,\hat{x_i},\cdots,\hat{x_j},\cdots,x_m)\nonumber \\
=&\widetilde{ev}_{(G_1,f_1)}( \_,\cdots,\_)\cdot \widetilde{ev}_{(G_2,f_2)}(\_,\cdots,\_,)(-1)^\sigma(L_1\sqcup L_2)_*(x_i,x_j,x_1,\cdots,\hat{x_i},\cdots,\hat{x_j},\cdots,x_m) \label{ia2}
\end{align}
if $\psi(G,f,I,J)=:\Big((G_1,f_1),(G_2,f_2),L_1,L_2\Big)\in C_{k_1,k_2}^{g,n}$, recalling the notations from definition \ref{grdf}.
In \eqref{ia2} we plug the $x_l$'s belonging to $L_1$, that is such that $l\in L_1$ (among them $x_i$ at the first place) into the first factor and the $x_l$'s belonging to $L_2$ (among them $x_j$ at the last place) into the second factor, indicated by $(L_1\sqcup L_2)_*$.
The sign $(-1)^\sigma$ arises from the resulting Koszul sign of the permutation of the $x_l$'s. There is no additional $\gamma$ factors in \eqref{ia2} since on the right hand side we both increase the number of total vertices by one, but also the number of connected components from one to two, thus the total betti number stays the same.
\end{obs}
Next we note that the indexing set of the second and third sum in \eqref{E} is in bijection with the set $A_{k_1,k_2}^{g,n}$. Theorem \ref{gt} tells us that this set is in bijection with $B_{k_1,k_2}^{g,n}\cup C_{k_1,k_2}^{g,n}$ and equations \eqref{ia1} and \eqref{ia2} imply that the evaluation maps are compatible with this bijection. Thus we have  
\begin{align*}
E=&\sum_{g,n}\sum_{B^{g,n}_{k_1,k_2}}\epsilon(i,1,..,\hat{i},..,\hat{j},.., n,j)\widetilde{ev}_{(G,f)}(x_i,x_1\cdots,\hat{x_i},\cdots,\hat{x_j},\cdots,x_m,x_j)\\
+\sum_{g,n}&\sum_{C^{g,n}_{k_1,k_2}}\widetilde{ev}_{(G_1,f_1)}(\_, ,\cdots,\_)\cdot \widetilde{ev}_{(G_2,f_2)}(\_,,\cdots,\_)(-1)^\sigma(L_I\sqcup L_J)_*(x_i,x_j,x_1,\cdots,\hat{x_i},\cdots,\hat{x_j},\cdots,x_m).
\end{align*}
Finally using the identification \eqref{obv2} and the fact that the evaluation maps are compatible with that identification we rewrite the first summand
\begin{align*}
E=&\gamma\sum_{g,n}\sum_{\widetilde{B}^{g,n}_{k_1,k_2}}\widetilde{ev}_{(G,f)}(x_1,\cdots ,x_m)\\
+&\sum_{g,n}\sum_{\widetilde{C}^{g,n}_{k_1,k_2}}\widetilde{ev}_{(G_1,f_1)}(\_,\cdots,\_)\cdot \widetilde{ev}_{(G_2,f_2)}(\_, \cdots,\_)(-1)^{\widetilde\sigma}(L_1\sqcup L_2)_*(x_1,\cdots,x_m).
\end{align*}
and where we have rewritten the second summand using identification \eqref{obv} and the fact that the evaluation maps are compatible with that. Now using again \eqref{id1} we deduce that
\begin{align*}
E=\gamma\sum_{g,n}\sum_{ (G,f)\in \widetilde{\Gamma(g,n)}_m}
  \frac{1}{|Aut(G,f)|}\cdot& K_{(G,f)}(x_1,\cdots,x_m)+\\
\sum_{g,n}\sum_{\substack{n_1+n_2=n\\g_1+g_2=g\\
m_1+m_2=m}}\sum_{ \substack{G_1\in {\Gamma(g_1,n_1)}_{m_1}\\
G_2\in {\Gamma(g_2,n_2)}_{m_2}}}&\sum_{\substack{L_1\sqcup L_2=[m] \\  i\in L_1,j\in L_2}}
  \\
  \frac{1}{|Aut(G_1)|\cdot |Aut(G_2)|}&\big(\mathcal{K}_{|L_1|}(\_,\cdots,\_)\cdot\mathcal{K}_{|L_2|}(\_,\cdots,\_)\big)(-1)^{\widetilde{\sigma}}(L_1\sqcup L_2)_*(x_1,\cdots,x_m)
  \end{align*}
which by definition is equal to  
  $$=\gamma\mathcal{K}_m(x_1\odot\cdots \odot x_m)+\sum_{\substack{L_1\sqcup L_2=[m] \\  i\in L_1,j\in L_2}}\epsilon(L_1,L_2)\big(\mathcal{K}_{|L_1|}(\bigodot_{i\in L_1} x_i)\cdot\mathcal{K}_{|L_2|}(\bigodot_{i\in L_2} x_i)\big),$$
  which is what we wanted to show.
\end{proof}
\begin{lemma}\label{easy}
For all $m>1$ and $(x_1\otimes y_1)\odot\cdots \odot (x_m\otimes y_m)\in Sym^m(\mathcal{F}^{c}(\mathcal{C})\otimes W[6-2d])$ we have\footnote{Here $(-1)^{*_1}=(-1)^{|y_i|}$ and $(-1)^{*_2}=(-1)^{\sum_{i\in L_1}|y_i|}$} 

\begin{align*}
&\sum_{i\neq j}\sum_{\substack{L_1\sqcup L_2=[m]\\ i\in L_1,j\in L_2}}\big(\mathcal{K}_{|L_1|}(\bigodot_{i\in L_1} x_i)\cdot\mathcal{K}_{|L_2|}(\bigodot_{i\in L_2} x_i)\big)\otimes(-1)^{*_1}\big(\{y_i,y_j\}_w\cdot  y_1\cdot\ldots \hat{y_i}\ldots\hat{y_j}\ldots\cdot y_m\big)\\
&=\sum_{\substack{L_1\sqcup L_2=[m] \\ |L_1|,|L_2|\geq 1}}\big(\mathcal{K}_{|L_1|}(\bigodot_{i\in L_1} x_i)\cdot\mathcal{K}_{|L_2|}(\bigodot_{i\in L_2} x_i)\big)\otimes (-1)^{*_2}\{\bigodot_{i\in L_1} y_i,\bigodot_{j\in L_2} y_j\}_w
\end{align*}
\end{lemma}
\begin{proof}
This follows directly from plugging the identity
$$(-1)^{\sum_{i\in L_1}|y_i|}\{\bigodot_{i\in L_1} y_i,\bigodot_{j\in L_2} y_j\}_w=\sum_{i\in L_1,j\in L_2}(-1)^{|y_i|}\{y_i,y_j\}_w\cdot  y_1\cdot\ldots \hat{y_i}\ldots\hat{y_j}\ldots\cdot y_m,$$
which follows from the Leibniz rule for $W$, into the second line, which gives the first line.
\end{proof}
Now we can easily give the proof of the key lemma \ref{key}:
\begin{proof}
For $m>1$ and $(x_1\otimes y_1)\odot\cdots \odot (x_m\otimes y_m)\in Sym^m(\mathcal{F}^{c}(\mathcal{C})\otimes W[6-2d])$ we have\footnote{Here the signs $a$, $b$, $c$ are as in the statement of lemma \ref{key}. Further $(-1)^d=(-1)^{\sum_{i\in L_1}|y_i|}$ times the Koszul sign of permuting the $x$'s to the left and $y$'s to right.} 
\begin{align*}
&\sum_{ i\neq j}(-1)^{*_a}\mathcal{K}_{m-1}\big((x_i\cdot x_j)\odot x_1\odot\ldots \hat{x_i}\ldots\hat{x_j}\ldots\odot x_m\big)\otimes\big(\{y_i,y_j\}_w\cdot  y_1\cdot\ldots \hat{y_i}\ldots\hat{y_j}\ldots\cdot y_m\big) \\
=&\sum_{i\neq j}(-1)^{*_b}\gamma\mathcal{K}_m(x_1\odot\cdots \odot x_m)\otimes\big(\{y_i,y_j\}_w\cdot  y_1\cdot\ldots \hat{y_i}\ldots\hat{y_j}\ldots\cdot y_m\big)\\
&+\sum_{i\neq j}\sum_{\substack{L_1\sqcup L_2=[m] \\  i\in L_1,j\in L_2}}(-1)^{*_d}\big(\mathcal{K}_{|L_1|}(\bigodot_{i\in L_1} x_i)\cdot\mathcal{K}_{|L_2|}(\bigodot_{i\in L_2} x_i)\big)\otimes\big(\{y_i,y_j\}_w\cdot  y_1\cdot\ldots \hat{y_i}\ldots\hat{y_j}\ldots\cdot y_m\big)\\
=&\sum_{i\neq j}(-1)^{*_b}\gamma\mathcal{K}_m(x_1\odot\cdots \odot x_m)\otimes\big(\{y_i,y_j\}_w\cdot  y_1\cdot\ldots \hat{y_i}\ldots\hat{y_j}\ldots\cdot y_m\big)\\
&+\sum_{\substack{L_1\sqcup L_2=[m] \\ |L_1|,|L_2|\geq 1}}(-1)^{*_c}\big(\mathcal{K}_{|L_1|}(\bigodot_{i\in L_1} x_i)\cdot\mathcal{K}_{|L_2|}(\bigodot_{i\in L_2} x_i)\big)\otimes \{\bigodot_{i\in L_1} y_i,\bigodot_{j\in L_2} y_j\}_w
\end{align*}
Here the first equality follows from theorem \ref{BVinf} (to be more precise from multiplying the equality stated by $(-1)^{|y_i|}$ times the Koszul sign from permuting in $(x_1\otimes y_1)\odot\cdots \odot (x_m\otimes y_m)$ the $x$'s to the left and $y$'s to right ). The second equality follows from lemma \ref{easy} (again multiplied by that same Koszul sign).
\end{proof}  
With this we have now proven theorem \ref{cwd}; let us quickly recall why:
\begin{Zusa}\label{Zusa}
As explained (beneath its statement) proving theorem \ref{cwd} is equivalent to proving equation \ref{inter}. By equations \ref{BDr} and \ref{KiL} proving equation  \ref{inter} is equivalent to proving lemma \ref{key}. Finally theorem \ref{BVinf} and the easy lemma \ref{easy} together give a proof of lemma \ref{key}, as we just explained above.
\end{Zusa}
A direct corollary is:
\begin{corollary}\label{ec}
Given a dimension $d$ cyclic $A_\infty$-category $\mathcal{C}$, a splitting $s$ of the non-commutative Hodge filtration and $\Lambda,$ a full subcategory of $\mathcal{C}$. Then there is an $L_\infty$ quasi-isomorphism
$$\mathcal{K}_s\otimes m:\ \mathcal{F}^{c}(\mathcal{C})\otimes \mathcal{F}^{o}(\Lambda)[5-2d]\rightsquigarrow \mathcal{F}^{c}(\mathcal{C})^{triv}\otimes \mathcal{F}^{o}(\Lambda)[5-2d].$$
\end{corollary}
\begin{proof}
    We apply theorem \ref{cwd} to the $(2d-5)$-twisted BD algebra $\mathcal{F}^{o}(\Lambda)$ from definition \ref{oBD}.
\end{proof}
Given a dimension $d$ cyclic $A_\infty$-category $\mathcal{C}$ and a splitting $s$ there is an isomorphism of induced dg Lie algebras (recalling definitions \ref{cBD_t} and \ref{cBD_Tr})
\begin{equation}\label{R}
R:\mathcal{F}^{c}(\mathcal{C})^{triv}\rightarrow \mathcal{F}^{c}(\mathcal{C})^{Triv},\end{equation}
see beginning of section 9.1 of \cite{CaTu24}, keeping in mind that we denoted eg. $\mathfrak{h}_{\mathcal{C}}^{triv}=\mathcal{F}^{c}(\mathcal{C})^{triv}.$ Thus it follows:
\begin{corollary}\label{maco}
Given a dimension $d$ cyclic $A_\infty$-category $\mathcal{C}$, a splitting $s$ of the non-commutative Hodge filtration and $\Lambda,$ a full subcategory of $\mathcal{C}$. Then there is an $L_\infty$ quasi-isomorphism
$$\Psi_s^{oc}:\ \mathcal{F}^{c}(\mathcal{C})\otimes \mathcal{F}^{o}(\Lambda)[5-2d]\rightsquigarrow \mathcal{F}^{c}(\mathcal{C})^{Triv}\otimes \mathcal{F}^{o}(\Lambda)[5-2d],$$
given by composing the isomorphism \eqref{R} with the $L_\infty$-map from corollary \ref{ec}.
\end{corollary}
Thus we have proven theorem A from the introduction.
\bibliography{refs}
\bibliographystyle{alpha}
\end{document}